\documentclass[11pt, a4paper]{amsart}

\usepackage{dsfont}
\usepackage{amsmath, amsthm, amssymb}
\usepackage{xypic}
\usepackage{graphicx}
\usepackage{color}
\usepackage[usenames,dvipsnames]{xcolor}
\usepackage[latin1]{inputenc}
\usepackage{indentfirst}
\usepackage{epstopdf}
\usepackage{subfigure}
\usepackage{emptypage}
\usepackage{hyperref}
\hypersetup{
    colorlinks=true, %set true if you want colored links
    linktoc=all,     %set to all if you want both sections and subsections linked
    linkcolor=blue,  %choose some color if you want links to stand out
    citecolor=JungleGreen,
}

\newtheorem{theorem}{Theorem}%[section]

\newtheorem{corollary}[theorem]{Corollary}
\newtheorem{example}[theorem]{Example}
\newtheorem{proposition}[theorem]{Proposition}
\newtheorem{lemma}[theorem]{Lemma}

\newtheorem{theorem*}{Theorem}
\newtheorem{question*}[theorem*]{Question}
\newtheorem{conjecture*}[theorem*]{Conjecture}
\newtheorem{corollary*}[theorem*]{Corollary}
\newtheorem{theorem*e}{Teorema}
\newtheorem{question*e}[theorem*e]{Pregunta}
\newtheorem{conjecture*e}[theorem*e]{Conjetura}
\newtheorem{corollary*e}[theorem*e]{Corolario}

\newcommand{\Fix}{\mathrm{Fix}}
\newcommand{\Per}{\mathrm{Per}}

\renewcommand{\int}{\mathrm{int}}

\newcommand{\id}{\mathrm{id}}
\newcommand{\dist}{\mathrm{dist}}

\newcommand{\vv}[1]{\overrightarrow{#1}}

\title{Indices of fixed points not accumulated by periodic points}

\author{Luis Hernández--Corbato}
\address{IMPA, Estrada dona Castorina 110, Rio de Janeiro, Brazil.}
\email{luishcorbato@mat.ucm.es, lcorbato@impa.br}

\thanks{The author has been supported by CNPq (Conselho Nacional de Desenvolvimento Científico  e Tecnológico - Brasil)
and partially by MICINN grant MTM2012-30719. The author also acknowledges kind support received from IMPA}



\begin{document}

\begin{abstract}
We prove that for every integer sequence $I$ satisfying Dold relations there exists
a map $f : \mathds{R}^d \to \mathds{R}^d$, $d \ge 2$, such that $\Per(f) = \Fix(f) = \{o\}$, where $o$ denotes the origin,
and $(i(f^n, o))_n = I$.
\end{abstract}

\maketitle

\section{Introduction}

Given a map $f$ defined in a Euclidean space onto itself and $p$ a fixed point of $f$, the fixed point index
or Lefschetz index of $f$ at $p$, denoted $i(f,p)$,
is an integer which measures the multiplicity of $p$ as a fixed point of $f$.
The definition requires the point $p$ to be isolated in the set of fixed points of $f$, which will be denoted $\Fix(f)$.
The index is a topological invariant of the local dynamics around $p$. Since a fixed point
of a map is also fixed by any of its iterates $f^n$, $n \ge 1$, the integer $i(f^n, p)$ is defined
as long as $p$ remains isolated in $\Fix(f^n)$. The integer sequence $(i(f^n, p))_{n = 1}^{\infty}$
will be denominated fixed point index sequence throughout this article.
In general, it is very difficult to find constraints for these invariants.
In fact, the unique global rule satisfied by fixed point index sequences is encompassed in the so--called
Dold relations, see \cite{dold}, which are described in Section \ref{sec:preliminaries}.
One of the most complete references on fixed point index theory is \cite{marzantowicz}.

In dimension 1, the only possible values of the index of a fixed point are -1, 0 and 1. From dimension 2 and on
any integer sequence satisfying Dold relations may appear as a fixed point index sequence of some map.
Some restrictions appear as we impose extra conditions over the map $f$. For instance, Shub and Sullivan proved in
\cite{shubsullivan} that the sequence is periodic when $f$ is $C^1$.
Recently, in \cite{smooth} Graff, Jezierski and Nowak--Przygodzi have given a complete description in the $C^1$ case.
Further, if $f$ is a homeomorphism of a surface
$(i(f^n, p))_n$ follows a very restrictive periodic pattern, see for example \cite{brown, lecalvezyoccoz,
lecalvezyoccozconley, rportalsalazar, bonino}. Periodicity of the sequence has been found to be true
in dimension 3 just for homeomorphisms and locally maximal fixed points
(i.e., points $p$ which have a neighborhood $V$ such that $\{p\}$ is the maximal invariant set in $V$), see
\cite{london, gt}.

Any of the previously described constraints disappear when the hypothesis are substantially weakened. For instance,
in the planar case if $f$ is no longer invertible then Dold relations are the only constraints remaining.
In \cite{graff}, Graff and Nowak--Przygodzki showed how to define a map in the plane fixing the origin and
such that the fixed point index sequence is a given integer sequence satisfying Dold relations.
Their map is constructed by gluing pieces made up of small radial sectors carrying a prescribed dynamics.
This operation produces lots of periodic points which accumulate in the fixed one.
Incidentally, notice that, in contrast, if the map is a homeomorphism and the fixed point is accumulated by $\Per(f)$ but
not by $\Fix(f^n)$ then $i(f^n, p) = 1$ (see \cite{pelikanslaminka} and also \cite{lecalvezENS}, pag. 145).

It is somehow surprising that if the fixed point $p$ is locally maximal (in the sense previously described)
and $f$ is a continuous map in the plane, the fixed point index sequence satisfies the following three constraints
(see \cite{dcds, graffpaco}):
$i(f, p)$ is bounded from above by 1, the sequence is periodic and every $a_k$ is non--positive for $k \ge 2$ (see Section
\ref{sec:preliminaries} for a definition of $a_k$). It is not known to what extent these constraints remain valid.
In this work we consider the hypothesis of isolation as a periodic point, which is
halfway between the locally maximal hypothesis and the unrestricted case.
In the case of homeomorphisms the behavior of the fixed point index under this hypothesis is very well
understood and similar to the locally maximal case (see \cite{brown, lecalvezENS, leroux, rportalsalazar2}).
However, we prove that for continuous maps it turns out that this weakening is enough to dissipate all three constraints:

\begin{theorem}\label{thm}
For every $d \ge 2$ and every integer sequence $I$ satisfying Dold relations
there exists a map $f : \mathds{R}^d \to \mathds{R}^d$ fixing a point $p$ such that
\begin{itemize}
\item $I = (i(f^n, p))_n$ and
\item $p$ is not accumulated by other periodic orbits of $f$.
\end{itemize}
\end{theorem}

The most interesting case included in this result is $d = 2$. For larger dimensions, it suffices to fix the map
in a plane and then retract the ambient space to that plane.

The paper is organized as follows. First, we introduce several definitions and comments along with some examples
of dynamics which are the basic pieces of our work. In Section \ref{sec:construction}, we carry out the construction of the map
which proves Theorem \ref{thm}. Our work relies on a definition whose discussion involves symbolic dynamics and is
postponed to the last section.

\section{Fixed point index}\label{sec:preliminaries}

Before we start with the basic definitions, let us go quickly through the notations and conventions used in the text.
The origin of the plane $\mathds{R}^d$ is denoted $o$ and we shall often use polar coordinates in $S^1 \times \mathds{R}$
to represent the puctured plane $\mathds{R}^2 \setminus \{o\}$. Notice a small annoyance, for later convenience
the radial coordinate takes values in $\mathds{R}$ instead of $\mathds{R}^+$ and the origin corresponds to
the end $r = - \infty$. %The curves $r = $ constant remain regular circles.
We will use two conventions for the angular coordinate: it will take values either in $(-\pi, \pi]$
or in $[0,1)$. Angles (points in $S^1$) will be denoted by greek letters.
Closed arcs in $S^1$ are named intervals and each interval $J$ determines a (radial) sector in the punctured
plane which contains all the points $(\theta, r)$ such that $\theta \in J$.
The (forward) orbit of a point $x$ under a map $F$ is the set $\{x, F(x), F^2(x), \ldots\}$.
Sequences appear ubiquitously and are always indexed by the positive integers.

Let $U$ be an open subset of $\mathds{R}^d$ and $f : U \to \mathds{R}^d$ a continuous map. Suppose that $p$
is an isolated fixed point of $f$. The fixed point index of $f$ at $p$, denoted $i(f, p)$, is defined as
the degree of the map $\id - f : U \to \mathds{R}^d$ at $p$. In other words, if $B$ is a closed ball centered at $p$
and such that $\Fix(f) \cap B = \{p\}$ then $i(f, p)$ is the degree of the map
\begin{align*}
\phi : & \partial B \to S^{d-1} \\
       &  x \mapsto \frac{x - f(x)}{||x - f(x)||}
\end{align*}

As our considerations will soon be limited to dimension 2, we include a more geometrical approach to the fixed point index
in the plane.
Let $f : \mathds{R}^2 \to \mathds{R}^2$ be a continuous map and $\gamma : S^1 \to \mathds{R}^2$ be a Jordan curve disjoint
from $\Fix(f)$. One can define another curve $\eta : S^1 \to \mathds{R}^2$
by $\eta(t) = \gamma(t) - f(\gamma(t))$. The index of $f$ along $\gamma$ is defined as the winding
number of $\eta$ around the origin.
%or, in other words, the topological degree of the map
%\begin{align*}
%\chi : & S^1 \to S^1 \\
%       & t \mapsto \frac{\eta(t)}{||\eta(t)||} \\
%\end{align*}
The definition does not depend on the parametrization of the curve $\gamma$ and we will often use $\gamma$ to refer
to the image of the curve as well.

Let $p$ be an isolated fixed point of $f$ and $V$ be an open neighborhood of $p$ which does not contain any other fixed point.
The index of $f$ along any Jordan curve $\gamma$ contained in $V$ which winds around $p$ is %independent of $\gamma$.
equal to the fixed point index of $f$ at $p$. The closed topological disk bounded by $\gamma$ plays the role of $B$
in the definition.
%This integer is called index of $f$ at $p$ or fixed point index of $p$ and is denoted $i(f, p)$.
It is not necessary that $\gamma$ is confined to $V$: if $D$ is a Jordan domain in the plane and $p$ is the only
fixed point contained in $D$ (more precisely, in the interior of $D$) the index of $f$ along the curve
$\partial D$ is again $i(f, p)$.

%Let us include a trivial remark. A point $p$ fixed by a map $f$ is also fixed by any of its iterates $f^n$, $n \ge 1$.
%The index of $f^n$ at $p$ is defined as long as $p$ is isolated in $\Fix(f^n)$.
%We are interested in the behavior of the integer sequence $(i(f^n, p))_n$, every one of its terms
%is defined as long as $p$ is isolated in $\Fix(f^n)$ for every $n \ge 1$. In particular, this is the case
%when $p$ is isolated in the set of periodic points of $f$, $\Per(f)$.

Henceforth, we assume $d = 2$ unless otherwise stated.

\begin{lemma}\label{lem:indexnotparallel}
Let $p$ be an isolated fixed point of the map $f$ and $\gamma$ be a Jordan curve which winds (in the positive sense) around $p$ and
does not enclose any other fixed point. Suppose the vectors $\vv{px}$ and $\vv{xf(x)}$ point in different
directions at every point $x$ in $\gamma$. Then, $i(f, p) = 1$.
\end{lemma}
\begin{proof}
The vector $\vv{px}$ makes one positive turn as $x$ moves along $\gamma$ and comes back to the starting point.
Since $\vv{xf(x)}$ does not overlap with $\vv{px}$ it must also make one positive turn as $x$ moves along $\gamma$
so, by definition, the index is 1.
\end{proof}

In view of Lemma \ref{lem:indexnotparallel} the most critical parts are those around the places
where the vectors $\vv{px}$ and $\vv{xf(x)}$ point in the same direction.
Describing the dynamics in such pieces will be enough to compute the fixed point index.

\begin{corollary}\label{cor:indexattractor}
Let $\gamma$ be the boundary of a closed disk $D$ centered at $p$ which does not contain any other point fixed by $f$.
If $f(\gamma) \subset D$ then $i(f, p) = 1$.
\end{corollary}

\begin{example}
Consider the planar map which fixes the origin and is given in polar coordinates by $(\theta, r) \mapsto (\theta, r - 1)$,
$\theta \in S^1, r \in \mathds{R}$.
Let $\gamma$ be the curve $r = 0$. By Corollary \ref{cor:indexattractor},  the index of the origin is 1.
Let us modify the map in a radial sector so that the dynamics becomes more interesting.
Assume now that $\theta$ takes values in $(-\pi, \pi]$ and define $S = \{(\theta, r): \theta \in [-1,1]\}$
and a map in $S$ by $f_-(\theta, r) = (\theta \cdot c_-(\theta), r + 1 - 2\theta^2)$,
where $c_- : [-1, 1] \to [1/2, 1]$ satisfies $c_-^{-1}(1) = \{-1, 0, 1\}$ and $\theta \cdot c_-(\theta)$ is strictly
increasing in $\theta \in [-1, 1]$.
See Figure \ref{fig:ejemplos}.
%Set $S = \{(\theta, r): \theta \in [-1,1]\}$ and define a map in $S$ by
%$f_-(\theta, r) = (\theta^3, r + 1 - 2\theta)$. Notice that we are here assuming that the angular
%coordinate $\theta$ ranges in the interval $(-\pi, \pi]$.
Using the original map, $f_-$ extends to a map in the whole plane.
As $x$ moves along the arc $\gamma \cap S$ the vector $\vv{xf(x)}$ makes one turn in the negative sense
so the index must be adjusted by subtracting $1$, $i(f_-, o) = 1 - 1 = 0$.

Note that the vectors $\vv{ox}$ and $\vv{xf_-(x)}$ point in the same direction only at $x = (\theta = 0, r = 0)$ and
the angular dynamics around $\theta = 0$ forces the contribution to the index to be negative,
as $\theta = 0$ is attracting for $\theta \mapsto \theta^3$.

Since any positive iterate $f_-^n$, $n \ge 1$, is topologically conjugate to $f_-$ we conclude that
$i(f_-^n, o) = i(f_-, o) = 0$.

%If we replace the map $f_-$ by any of its positive iterates $f_-^n$, $n \ge 1$, the new dynamics is very similar
%to the original one (actually, they are conjugate) and it is not difficult to check that $i(f_-^n, o) = i(f_-, o) = 0$. Indeed, %the
%description contained in the previous paragraph translates verbatim to any positive iterate of $f_-$ as well.

A minor modification in the angular behavior of $f_-$ at $\theta = 0$ reverses the sign of its extra contribution to the index.
Choose $c_+ : [-1, 1] \to [1, 2]$ such that $c_+^{-1}(1) = \{-1, 0, 1\}$ and $\theta \cdot c_+(\theta)$ is
strictly increasing in $\theta \in [-1, 1]$.
Define $f_+(\theta, r) = (\theta \cdot c_+(\theta), r + 1 - 2\theta^2)$ in $S$ and $f_+(\theta, r) = (\theta, r-1)$
otherwise. Now the vector $\vv{xf_+(x)}$ makes one turn in the positive sense as $x$ moves along $\gamma \cap S$
 so $i(f_+, o) = 1 + 1 = 2$.
Notice that the modification in the angular coordinate has reversed the angular dynamics around $\theta = 0$,
it is now repelling.
%Again, the behavior of any iterate $f_+^n$ is essentially the same as far as the index is concerned.
Again, since $f_+^n$ is conjugate to $f_+$ we obtain that $i(f_+^n, p) = i(f_+, p) = 2$ for every $n \ge 1$.
\end{example}

\begin{figure}[htb]
\begin{center}
\begin{tabular}{l c r}
\includegraphics[scale = 0.7]{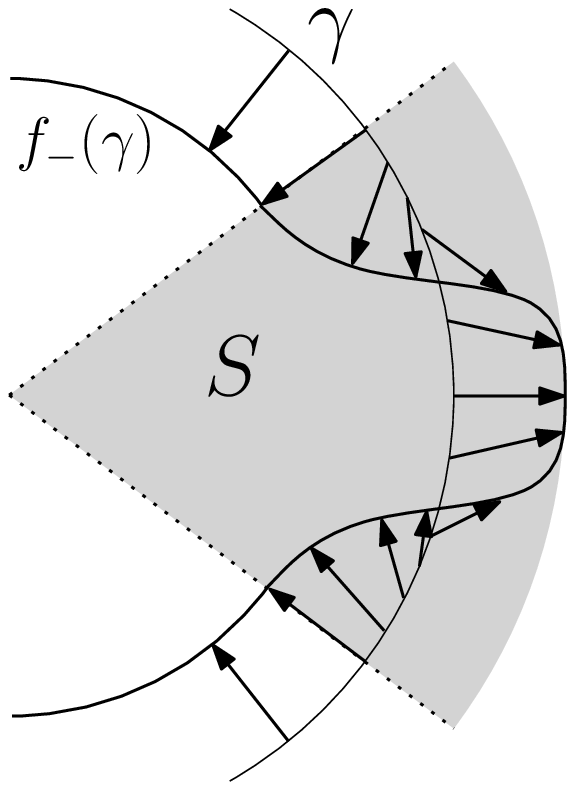} & \hspace{2cm} & \includegraphics[scale =
0.7]{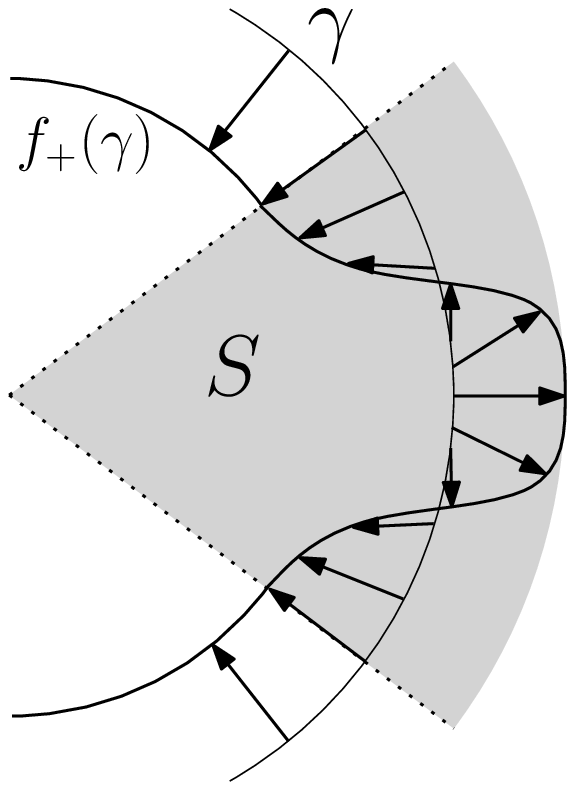}
\end{tabular}
\end{center}
\caption{Qualitative description of maps $f_-$ (left) and $f_+$ (right) within the sector $S$.}
\label{fig:ejemplos}
\end{figure}

\begin{example}\label{ex:iterates}
Let us show a way to modify arbitrarily the index using the previous construction.
Fix an integer $m \ge 1$ and start with the map $(\theta, r) \mapsto (\theta, r - 1)$ as before.
Take any radial sector $S_m$ and divide it into $m$ equal sectors $S_m = T_1 \cup \ldots \cup T_m$. Set the map in each sector $T_i$
to be equal to $f_- : S \to S$, after a suitable affine transformation in the angular coordinate, and
denote $f_{m,-} : S_m \to S_m$ the resulting map.

It is not difficult to check that the origin has now index $i(f_{m,-}, o) = 1 - m$ as the vector
$\vv{xf_{m,-}(x)}$ winds $m$ times in the negative sense as $x$ moves along $\gamma \cap S_m$.

In an analogous fashion one can define $f_{m,+} : S_m \to S_m$ and show that $i(f_{m,+}, o) = 1 + m$.
\end{example}

We finish this section with some general considerations about the fixed point index sequence $I = (i(f^n, p))_n$.
Albrecht Dold \cite{dold} found the unique constraint satisfied by any fixed point index sequence.
Define the normalized sequences $\sigma^k = (\sigma^k_n)_n$ by
\begin{equation*}
\sigma^k_n =
\begin{cases}
k & \text{if } n \in k \mathds{N} \\
0 & \text{otherwise} \\
\end{cases}
\end{equation*}
Dold relations, also called Dold congruences, state that if $I$ is any fixed point index sequence then $I$ is a (formal) integer
combination of normalized sequences, i.e. there exist integers $a_k$, $k \ge 1$, such that
\begin{equation}\label{eq:integercombination}
I = \sum_{k \ge 1} a_k \sigma^k.
\end{equation}
As already commented in the introduction, Dold relations are the only general constraints governing
fixed point index sequences.

\section{Theorem \ref{thm}: Construction}\label{sec:construction}

Let us start with the proof of Theorem \ref{thm}.
Fix from now and on a sequence of integers $(a_k)_k$ which in view of Equation (\ref{eq:integercombination})
uniquely determines the target sequence of fixed point indices $(i(f^n, o))_n$.
We will construct a map $\tilde{f} : \mathds{R}^d \to \mathds{R}^d$ whose only periodic orbit is
the origin, which is a fixed point, and such that
\begin{equation}\label{eq:originreduced}
(i(\tilde{f}^n, o))_n = \sum_{k \ge 1} a_k \sigma^k.
\end{equation}

It is enough to prove the result for $d = 2$. Indeed, if $f$ is a planar map satisfying Theorem \ref{thm}
we split $\mathds{R}^d = \mathds{R}^2 \oplus \mathds{R}^{d-2}$ and define $\tilde{f} = (f, c)$, where $c$
denotes the map that sends every point to the origin. From the multiplicativity of the fixed point index
and the fact that $c$ has index 1 at the origin we obtain that $(i(\tilde{f}^n, o))_n = (i(f^n, o))_n$.
%Since $\Per(\tilde{f})$ is homeomorphic to $\Per(f)$,
Clearly $\tilde{f}$ satisfies the requirements of Theorem \ref{thm}.

Our target map $f: \mathds{R}^2 \to \mathds{R}^2$ will have the appearance of a skew--product in polar coordinates.
We start with a very simple dynamics in the base, the angular coordinate, namely
$e_2(\theta) = 2 \theta \enskip\mathrm{mod}\enskip 1$ (now we assume $\theta$ ranges in $[0, 1)$).
The dynamics will become richer after we replace the angles in a set of periodic orbits of $e_2$ by
intervals and extend the map to them. The set of periodic orbits $\Lambda$ in which this blow--up procedure
will be carried out is defined by the following two properties:
\begin{itemize}
\item[(i)] $\Lambda$ is composed of exactly one orbit of each period in a set $\mathcal{P} \subset \mathds{N}$.
\item[(ii)] $\Lambda' \cap \Per(e_2) = \emptyset$, i.e. no point of accumulation of $\Lambda$ is periodic under $e_2$.
\end{itemize}
The task of showing that the definition of $\Lambda$ is not vacuous is postponed to Section \ref{sec:symbolic},
where we produce a set satisfying both of its defining properties.
We use the notation $X'$ to refer to the derived set of a set $X$. It is composed of all accumulation points, also called
limit points, of $X$.
%In our case, $\Lambda' = \emptyset$ if and only if $\Lambda$ is finite and, equivalently, so is $\mathcal{P}$.
%Otherwise, $\Lambda$ is countable hence its derived set is non--empty.
%The fact that no point of $\Lambda'$ is periodic under $e_2$ is crucial in our construction.
Property (i) ensures that $\Lambda'$ is invariant under $e_2$.
Note that (ii) implies that any $x \in \Lambda$ is isolated in $\Lambda$.

First, we begin with a map $f_0$ whose radial dynamics makes all the indices $i(f_0^n, o)$ to be equal to 1.
We group in $\mathcal{P}$ the set of exponents $n$
for which we need to modify the map to achieve the desired value for the index of $f^n$ which,
in view of (\ref{eq:originreduced}), is
\begin{equation}\label{eq:indicesorigin}
i(f^n, o) = \sum_{k | n} k a_k.
\end{equation}
Thus, we set that 1 belongs to $\mathcal{P}$ iff $a_1 \neq 1$ and $k \in \mathcal{P}$ iff $a_k \neq 0$ for $k > 1$.

The map $f_0$ is given in polar coordinates by
$$f_0: (\theta, r) \mapsto (e_2(\theta), r - g_0(\theta)).$$
The stretching in the radial direction is determined by the map $g_0: S^1 \to \mathds{R}$ which we define by
$g_0(\theta) = \dist(\theta, \Lambda')$. Note that (ii) guarantees that the orbit under $f_0$ of
every point whose angular coordinate is periodic under $e_2$ tends to the origin.
However, if $\Lambda' \neq \emptyset$ there are points whose $\omega$-limit is not $\{o\}$ and they are not periodic.

An application of Corollary \ref{cor:indexattractor} shows that $i(f_0^n, o) = 1$ for any $n \ge 1$.
The dynamics of $f_0$ is somehow complicated though. For any $r_0$, $\Lambda' \times \{r_0\}$ is a closed invariant set
contained in the circle $\{r = r_0\}$ %centered at the origin
which does not contain any periodic point. The origin is then far from being locally maximal,
since the maximal invariant set contained in the closed disk $\{r \le r_0\}$
is $(\Lambda' \times \{r \le r_0\}) \cup \{o\}$.

Now, we enrich the dynamics of $e_2$ by replacing each point $\alpha \in \Lambda$ with a non--degenerate interval $J_{\alpha}$.
Let $\pi : S^1 \to S^1$ be the map which realizes this operation, $\pi^{-1}(\alpha) = J_{\alpha}$
 for every $\alpha \in \Lambda$ and $\pi$ is a homeomorphism outside the intervals $J_{\alpha}$.
Then, $e_2$ lifts to a map $h : S^1 \to S^1$ which satisfies $\pi \circ h = e_2 \circ \pi$.
%and we impose that the restriction of $h_{|J_{\alpha}}: J_{\alpha} \to J_{e_2(\alpha)}$ is an affine map.
Define accordingly the map $g$ in $\pi^{-1}(S^1 \setminus \Lambda)$ so that $\pi \circ g = g_0 \circ \pi$
and set
$$f: (\theta, r) \mapsto (h(\theta), r - g(\theta)).$$
It is intentionally still left to complete the definition of $g, h$ for angles inside the intervals $J_{\alpha}$
and thus the definition of $f$ within the radial sectors $S_{\alpha}$ determined by them.

Let $\{\alpha_0, \ldots, \alpha_{n-1}\}$ be an orbit of period $n$ under $e_2$ contained in $\Lambda$.
The dynamics of $f$ in the sectors
$S_{\alpha_0}, \ldots S_{\alpha_{n-1}}$ will be responsible for the value of the integer coefficient $a_n$
in (\ref{eq:originreduced}).
We define $f_{|S_{\alpha_i}}: S_{\alpha_i} \to S_{\alpha_{i+1}}$ to be equal to,
after affine rescaling in the angular coordinate, $f_{m,*}$, where $* = + $ or $-$.
The choice of sign and integer $m$ is:

\begin{equation}\label{eq:msign}
(m, *) =
\begin{cases}
(a_1 - 1, +) & \text{if } n = 1 \text{ and } a_1 \ge 2 \\
(1 - a_1, -) & \text{if } n = 1 \text{ and } a_1 \le 0 \\
(a_n, +) & \text{if } n \ge 2 \text{ and } a_n > 0 \\
(-a_n, -) & \text{if } n \ge 2 \text{ and } a_n < 0. \\
\end{cases}
%\tag{$\star$}
\end{equation}

It is convenient to say a few words about the dynamics of $f$ before focusing on the index computation.
It still shares some of the properties of $f_0$. For instance, for every $r_0 \in \mathds{R}$
$\pi^{-1}(\Lambda') \times \{r_0\}$ is invariant and the maximal invariant set in the closed disk
$\{r \le r_0\}$ is $(\pi^{-1}(\Lambda') \times \{r \le r_0\}) \cup \{o\}$.
Each sector $S_{\alpha}$ is periodic and contains an odd number ($2m + 1$ exactly)
of periodic rays of the same period as $S_{\alpha}$,
which are alternatively attracting or repelling, the two rays in the boundary of $S_{\alpha}$ being attracting.
Here attracting means that the orbit of every point in the ray tends to the origin and repelling means
that it goes to infinity.
The orbit of any other point in $S_{\alpha}$ goes to infinity and approaches a repelling periodic ray
provided the sign in (\ref{eq:msign}) is ``-''. Otherwise, in case the plus sign was assigned
in (\ref{eq:msign}), most of the orbits tend to the origin and approach an attracting periodic ray.

In the complement of the regions already described the behavior of the orbits is not particularly simple. Orbits
of points in rays of angle $\theta$ such that $\theta \notin \pi^{-1}(\Lambda')$
 and whose orbits under $h$ never land in an interval $J_{\alpha}$
tend to the origin. However, any of the previously described asymptotics might occur if the orbit eventually lands
in a point with angle in $\pi^{-1}(\Lambda')$ or in some $J_{\alpha}$.

On top of all the previous discussion, let us emphasize that $\Per(f) = \{o\}$.
Indeed, in any orbit the radial coordinate only stabilizes if the angular coordinate eventually lies in $\pi^{-1}(\Lambda')$.
Since in $\pi^{-1}(\Lambda')$ there are no periodic points under $h$, the origin is an isolated periodic point.

\begin{proposition}\label{prop}
$$i(f^n, o) = \sum_{k | n} k a_k.$$
\end{proposition}
\begin{proof}
The computation of the index will be done using the curve $\gamma$, which is the boundary of a circle
centered at $o$ (anyone works fine).

Consider the collection $\mathcal{J}$ of intervals $J_{\alpha}$ which satisfy $h^n(J_{\alpha}) = J_{\alpha}$.
It contains every $J_{\alpha}$ such that $n$ is a multiple of the period of $\alpha \in \Lambda$.
Consider the union of the radial sectors $S_{\alpha}$ determined by all $J_{\alpha}$ in $\mathcal{J}$
and denote $C$ its complement in the plane. The picture in $C$ is easy to analyze.
Indeed, as a point $x$ moves along a component $J'$ of $\gamma \cap C$ the vector
$\vv{xf^n(x)}$ never points in the same direction as $\vv{ox}$.
Moreover, it starts and ends its tour along $J'$ pointing
to the opposite direction as $\vv{ox}$ because $g(\theta)$ is strictly negative in $\Lambda$ and the endpoints of $J'$ are fixed by $h^n$.
Incidentally, note that the arc $f^n(J')$ might wind around the origin several times but this behavior does not make any impact
in the index because $f^n(J')$ lies inside the disk enclosed by $\gamma$.

Thus, the problem amounts to examine what happens in every arc $\gamma \cap S_{\alpha}$ where $J_{\alpha} \in \mathcal{J}$.
Denote $(m, *)$ the pair associated to $S_{\alpha}$ according to (\ref{eq:msign}).
As noticed in Example \ref{ex:iterates}, all the maps $f_{m,*}^l$, $l \ge 1$, are conjugate.
%is identical for all $l \ge 1$, as far as the computation of the index concerns.
Up to an affine transformation in the angular coordinate, the restriction of $f^n$ to $S_{\alpha}$ is equal to
$f_{m,*}^l$, where $l = n/k$ and $k$ is the period of $\alpha$.
This remark allows to conclude that, as $x$ moves along $\gamma \cap S_{\alpha}$,
the vector $\vv{xf^n(x)}$ turns exactly $m$ times in the (positive or negative) sense given by $*$.

Consequently, the dynamics in sector $S_{\alpha}$ adds $a_1 - 1$ to the index if $k = 1$ and $a_k$ if $k \ge 2$.
We need to take care of $k$ sectors for each $k$ divisor of $n$ contained in $\mathcal{P}$ so we finally obtain
$$i(f^n, o) = 1 + (a_1 - 1) + \sum_{k \ge 2, k | n} k a_k$$
and the result follows.
\end{proof}

In order to finish the proof of Theorem \ref{thm} it suffices to notice that the result of Proposition \ref{prop}
simply rephrases Equation (\ref{eq:originreduced}).

\section{Symbolic dynamics: Definition of $\Lambda$}\label{sec:symbolic}

This section is devoted to discuss the definition of $\Lambda$. Difficulties arise when trying to meet
property (ii) and this makes the example we present here a bit involved.
%This section is devoted to discuss property (ii) in the definition of the set $\Lambda$, which is essential in the construction.
Notice that if $\Lambda$ did not satisfy (ii) there would exist points periodic points $\beta$ under $h$ accumulated
by intervals $J_{\alpha}$. Since the absolute value of $g$ is arbitrarily small in $J_{\alpha}$ as
the period of $\alpha$ grows, we would have $g(\beta) = 0$ and so every point in the ray $\{\theta = \beta\}$
would be periodic, contrary to our hypothesis.

Symbolic dynamics eases the analysis of properties from subsets of our system
whose dynamics can be encoded properly. We will construct $\Lambda$ from a subset of a symbolic dynamical system
defined using the Prouhet--Thue--Morse sequence. Several other approaches exist as well. For instance, one may
use a symbolic sequence for which the relative density of each symbol is irrational.

Let $\Sigma_2$ be the set of one--sided infinite sequences in two symbols $\{0,1\}$ and $\sigma$ the shift map in $\Sigma_2$.
The dynamical system $(S^1, e_2)$ is a factor of $(\Sigma_2, \sigma)$. The semiconjugation
$\pi : \Sigma_2 \to S^1$ first sends an infinite sequence $s = d_1 d_2\cdots d_n\cdots$
 to the number $x = \sum_n d_n 2^{-n}$ in $[0,1]$ which has $s$ as binary expansion and then projects it to $S^1$.
For the sake of clarity, here we view $S^1$ as the interval $[0, 1]$ whose endpoints are identified.
The following diagram is commutative

\begin{equation*}
\xymatrix{
\Sigma_2 \ar[d]^{\pi} \ar[r]^{\sigma} & \Sigma_2 \ar[d]^{\pi} \\
S^1 \ar[r]^{e_2} & S^1
}
\end{equation*}

%The map $\pi$ is injective except in the set of sequences whose tail is composed of either all 0's or all 1's.
%Thus, except for the constant sequences whose images under $\pi$ coincide, every other
%periodic point $\beta$ of $S^1$ under $e_2$ has a unique preimage $\pi^{-1}(\beta)$ which is an infinite sequence
%periodic under $\sigma$ having the same period as $\beta$.
%\revisar

Let $(s_n)_n$ be a sequence of periodic infinite sequences whose periods tend to infinity. Suppose that
$(s_n)_n$ has limit $s$ and $s$ is periodic, i.e. it is made up of the infinite repetition of some word $w$
(a finite sequence of 0's and 1's).
Then, the sequence $s_n$ starts with the word $w$ for sufficiently large $n$.
Furthermore, given any positive integer $k$ it must be the case that $s_n$ starts with the word $w^k$, where
$w^k$ denotes the word $ww\cdots w$ consisting of $w$ repeated $k$ times in a row.
Thus, if for some $k \ge 1$ we ensure that $s_n$ does not start with an instance of the pattern $w^k$,
for any word $w$ and large enough $n$, we prove that the limit of $(s_n)_n$ is not a periodic sequence.

Our goal is to find a set of periodic sequences $A$ in $\Sigma_2$ satisfying the following properties:
\begin{itemize}
\item[(a)] For every $n \ge 1$, $A$ contains a sequence of period $n$.
\item[(b)] $A$ is invariant under $\sigma$, that is, it contains the whole orbit of each of its periodic points.
\item[(c)] There exists an integer $k$ such that at most a finite number of elements of $A$ start with the pattern $w^k$.
\end{itemize}
Once we obtain $A$ we may simply set $\Lambda = \pi(A)$ and check properties (i)--(ii) from the previous section.
The semiconjugation guarantees that any point of $\Lambda$ is periodic. We shall remove from $\Lambda$ all orbits
whose periods do not belong to $\mathcal{P}$. The previous considerations and condition (c) ensure that
property (ii) is satisfied.

The set $A$ will be created using the Prouhet--Thue--Morse sequence, which is defined as follows. Start with the word 0
and at each step do the replacements $0 \mapsto 01$, $1 \mapsto 10$. The words built in the first stages are
$$0, \enskip 01, \enskip 0110, \enskip 01101001, \enskip 0110100110010110, \enskip \ldots$$
The process continues ad infinitum and the words converge (note that their starting subwords are equal)
to an infinite sequence \textbf{t} named after Prouhet, Thue and Morse, who discovered it independently,
$$\textbf{t} = 01101001100101101001011001101001\ldots$$
This sequence exhibits an aperiodic yet recurrent behavior and has shown up in various fields of mathematics.
One striking feature of $\textbf{t}$ makes it interesting to us: it is cube--free, that is it does not contain
any instance of the pattern $w^3$ (a word appearing three times in a row).
This result is part of one of the foundational works on the field of combinatorics on words and was first
proved by Axel Thue in 1912 \cite{thue} and then rediscovered by Marston Morse in 1921 \cite{morse}.
%In fact, the original result was a bit stronger: the sequence $t$ does not contain any overlap.
For example the sequences
$$10\textbf{010101}10100110, \enskip 010110\textbf{100100100}101100110 $$
are not cube--free because they contain the words $010101 = 01^3$ and $100100100 = 100^3$, respectively.

A word $v$ is said to be conjugate to $w$ if there are (possibly empty) words $x, y$ such that $w = xy$ and $v = yx$.
The circular word of $w$ is the set consisting of $w$ and all of its conjugates.
There is an evident one--to--one correspondence between periodic infinite sequences whose period is a divisor of $n$
and words of length $n$: any such sequence $s$ is formed by the repetition of a word $w$.
 %which may be expressed as simply as $a = ww\cdots w\cdots$. \revisar
As a tiny trivial remark, note that the orbit of $s$ under the action of the shift map is made up of the infinite sequences generated from each of the conjugates of the word $w$.
For example the circular word of $100$ is $\{100, 001, 010\}$ and the orbit of $s = 100100100\ldots$
under $\sigma$ has period 3:
$$\sigma(s) = 001001001\ldots, \enskip \sigma^2(s) = 010010010\ldots, \enskip \sigma^3(s) = 100100100 \ldots = s.$$

Consider again Prouhet--Thue--Morse infinite sequence
$$\textbf{t} = d_1 d_2 \ldots d_n \ldots$$
and, for every $n \ge 1$, define the word $s_n = d_1\ldots d_n$ of length $n$. Let $s_n^*$ be the infinite
sequence generated by repetition of $s_n$. Set $A$ to be the union of the orbits under $\sigma$ of all the sequences $s_n^*$.

\begin{proposition}
$A$ satisfies properties (a)--(c) above.
\end{proposition}
\begin{proof}
The first two properties follow from the definition so we shall concentrate on (c).
We will prove that no conjugate of $s_n^*$ starts with an instance of the pattern $w^6$ if $n$ is large enough.
In other words, there does not exist any word $w$ such that a conjugate of $s_n$ starts with $w^6$.
Suppose on the contrary that some conjugate $yx$ of $s_n = xy$ starts like that and $n > 6\;\mathrm{length}(w)$.
It follows that either $x$ or $y$ contains the word $w^3$. This fact leads to contradiction because both $x$ and $y$ are
subwords of the Prouhet--Thue--Morse sequence which is known to be cube--free.
\end{proof}

The previous method generates a set of circular words of every length which are 6--power free, i.e.
do not contain an instance of the patter $w^6$. The construction is very elementary and very far from
being optimal. For a more up--to--date account on this topic together with an optimal result we refer the reader to \cite{circular}.

\end{document}